\documentclass[12pt]{article}
\usepackage{times}
\usepackage{amsfonts, amsmath, amssymb, amsthm, mathtools, dsfont}
\usepackage[colorlinks=true,citecolor=blue]{hyperref}
\usepackage[margin=2.5cm]{geometry}
\usepackage[utopia]{mathdesign}
\usepackage[mathscr]{euscript}
\usepackage[utf8]{inputenc}
\usepackage{tikz,kotex}
\usetikzlibrary{arrows}
\usetikzlibrary{decorations.pathreplacing,decorations.markings}

\tikzset{
  on each segment/.style={
    decorate,
    decoration={
      show path construction,
      moveto code={},
      lineto code={
        \path [#1]
        (\tikzinputsegmentfirst) -- (\tikzinputsegmentlast);
      },
      curveto code={
        \path [#1] (\tikzinputsegmentfirst)
        .. controls
        (\tikzinputsegmentsupporta) and (\tikzinputsegmentsupportb)
        ..
        (\tikzinputsegmentlast);
      },
      closepath code={
        \path [#1]
        (\tikzinputsegmentfirst) -- (\tikzinputsegmentlast);
      },
    },
  },
  mid arrow/.style={postaction={decorate,decoration={
        markings,
        mark=at position .5 with {\arrow[#1]{stealth}}
      }}},
}

\usepackage{graphicx}
\usepackage{parskip}
\usepackage{enumerate}
\usepackage{verbatim}
\usepackage{units}

\usetikzlibrary{calc}
\usepackage[outline]{contour}
\contourlength{1.2pt}

\begingroup
    \makeatletter
    \@for\theoremstyle:=definition,remark,plain\do{%
        \expandafter\g@addto@macro\csname th@\theoremstyle\endcsname{%
            \addtolength\thm@preskip\parskip
            }%
        }
\endgroup

\newtheorem{theorem}{Theorem}[section]
\newtheorem*{theorem*}{Theorem}
\newtheorem{lemma}[theorem]{Lemma}
\newtheorem{prop}[theorem]{Proposition}

\newtheorem{cor}[theorem]{Corollary}

\newtheorem{ques}[theorem]{Question}

\theoremstyle{definition}

\newtheorem*{remark*}{Remark}
\newtheorem{claim}[theorem]{Claim}

\newtheorem*{ques2}{Question~\ref{que-main2}}
\newtheorem*{ques1}{Question~\ref{que-main}}

\newcommand{\df}[1]{{{\color{black}\bf\em #1}}}

\newcommand{\from}{\leftarrow}

\begin{document}

\title{On operations preserving semi-transitive orientability of graphs}

\author{
{Ilkyoo Choi}\thanks{
Supported by the Basic Science Research Program through
the National Research Foundation of Korea (NRF) funded by the Ministry of
Education (NRF-2018R1D1A1B07043049), and also by Hankuk University of Foreign
Studies Research Fund.
Department of Mathematics, Hankuk University of Foreign Studies, Yongin-si, Gyeonggi-do, Republic of Korea.
\texttt{ilkyoo@hufs.ac.kr}.
}
\and
Jinha Kim\thanks{
Corresponding author. 
Department of Mathematical Sciences, Seoul National University, Seoul, Republic of Korea.
\texttt{kjh1210@snu.ac.kr}
}
\and
Minki Kim\thanks{
Supported by the National Research Foundation of Korea (NRF) grant funded by the Ministry of Education (NRF-2016R1D1A1B03930998).
Department of Mathematics, Technion -- Israel Institute of Technology, Haifa, Israel.
\texttt{kimminki@campus.technion.ac.il}
}
}

\date\today

\maketitle

\begin{abstract}
We consider the class of semi-transitively orientable graphs, which is a much larger class of graphs compared to transitively orientable graphs, in other words, comparability graphs. 
Ever since the concept of a semi-transitive orientation was defined as a crucial ingredient of the characterization of alternation graphs, also known as word-representable graphs, it has sparked independent interest. 

In this paper, we investigate graph operations and graph products that preserve semi-transitive orientability of graphs. 
The main theme of this paper is to determine which graph operations satisfy the following statement: if a graph operation is possible on a semi-transitively orientable graph, then the same graph operation can be executed on the graph while preserving the semi-transitive orientability. 
We were able to prove that this statement is true for edge-deletions, edge-additions, and edge-liftings. 
Moreover, for all three graph operations, we show that the initial semi-transitive orientation can be extended to the new graph obtained by the graph operation. 

Also, Kitaev and Lozin explicitly asked if certain graph products preserve the semi-transitive orientability. 
We answer their question in the negative for the tensor product, lexicographic product, and strong product.
We also push the investigation further and initiate the study of sufficient conditions that guarantee a certain graph operation to preserve the semi-transitive orientability.
\end{abstract}

\section{Introduction}
All graphs in this paper are undirected and simple, which means no loops and no multiple edges. 
Given a graph $G$, let $V(G)$ and $E(G)$ denote the vertex set and the edge set of $G$, respectively. 
An \df{orientation} of a graph $G$ assigns to each edge $uv\in E(G)$ a direction, either from $u$ to $v$ or from $v$ to $u$. 
An edge $uv$ directed from $u$ to $v$ is denoted by $u\to v$.
A \df{directed graph}, or a \df{digraph}, is a graph with an orientation. 
An orientation $\phi$ of $G$ is said to be \df{acyclic} if there is no directed cycle in the directed graph $\phi(G)$, and $\phi$ is said to be \df{transitive} if for every three distinct vertices $v_1, v_2, v_3 \in V(G)$, there is a directed edge ${v_1\to v_3}$ in $\phi(G)$ whenever ${v_1\to v_2}$ and ${v_2\to v_3}$ exist in $\phi(G)$.
It is not hard to see that a transitive orientation of a graph is always acyclic.
An undirected graph admitting a transitive orientation is also known as a \df{comparability graph}.
Due to their deep connection with partially ordered sets, comparability graphs have become one of the most actively researched graph classes throughout the literature. 
Examples of comparability graphs are complete multipartite graphs, permutation graphs, cographs~\cite{Jung78}, and threshold graphs. 

In this paper, we consider a graph class that is more general than the class of comparability graphs, namely, graphs admitting semi-transitive orientations. 
The definition of a semi-transitive orientation was first given in the study of word-representable graphs (see Section~\ref{sec:word}), but we find the orientation of independent interest. 

Given an acyclic orientation $\phi$ of a graph $G$, a \df{shortcut} is a non-transitive induced subdigraph on $\{v_1,\dots,v_n\}\subset V(G)$ that contains a directed edge $v_1 \to v_n$ and a directed path $v_1\to \cdots \to v_n$ on at least four vertices. 
Note that if $S$ is a shortcut in $\phi(G)$, then the induced subgraph $G[V(S)]$ of $G$ is not complete.
The orientation $\phi$ is said to be \df{semi-transitive} if $\phi(G)$ does not contain a shortcut.
If one can find a semi-transitive orientation of $G$, then we say that $G$ has a semi-transitive orientation, or $G$ is \df{semi-transitively orientable}. 
Clearly, comparability graphs are semi-transitively orientable, and there exist infinitely many graphs with semi-transitive orientations that are not comparability graphs. 
For example, every cycle is semi-transitively orientable while every odd cycle of length at least $5$ does not have a transitive orientation.
Hence, the class of semi-transitively orientable graphs is a much larger class of graphs compared to the class of comparability graphs.

Observe that given a graph $G$ with an arbitrary semi-transitive orientation $\phi$, the induced subdigraph of $\phi(G)$ on a clique of $G$ must be transitive.
In fact, Kitaev and Pyatkin~\cite{KP2008} proved the following stronger statement that every semi-transitive orientation of a graph is ``locally'' transitive.
The set of neighbors of a vertex $v$ is denoted by $N(v)$.

\begin{theorem}[\cite{KP2008}]\label{nb}
If $v$ is a vertex of a semi-transitively orientable graph $G$, then the induced subgraph of $G$ on $N(v)$ is a comparability graph.
\end{theorem}

By Theorem~\ref{nb}, a graph $G$ does not have a semi-transitive orientation if there exists a vertex $v\in V(G)$ such that $G[N(v)]$ is not a comparability graph.
Since an odd cycle of length at least $5$ is not a comparability graph, a wheel graph on an even number of vertices, larger than $5$, does not have a semi-transitive orientation. 
Note that a wheel graph, or simply, a \df{wheel}, on $n+1$ vertices, denoted by $W_{n}$, is a graph where one vertex is adjacent to all vertices of a cycle on $n$ vertices. 
It is worth mentioning that the graph $W_5$ in Figure~\ref{wheel5} is a graph with the minimum number of vertices that is not semi-transitively orientable, as shown in~\cite{KP2008}.
In other words, every graph on at most $5$ vertices is semi-transitively orientable.
Moreover, all graphs on $6$ vertices other than $W_5$ have semi-transitive orientations~\cite{KL2015}; 
we state this as Theorem~\ref{thm:6vx}.

\begin{theorem}[\cite{KP2008,KL2015}]\label{thm:6vx}
A graph on at most $6$ vertices has a semi-transitive orientation if and only if it is not $W_5$. 
\end{theorem}

See~\cite{KL2015} for all $25$ graphs on $7$ vertices with no semi-transitive orientations. 

\begin{figure}[htbp]
\centering
\begin{tikzpicture}[main node/.style={fill,circle,draw,inner sep=0pt,minimum size=3pt}, scale=1]
\node[main node] (v1) at ({90+72*0}:2){};
\node[main node] (v2) at ({90+72*1}:2){};
\node[main node] (v3) at ({90+72*2}:2){};
\node[main node] (v4) at ({90+72*3}:2){};
\node[main node] (v5) at ({90+72*4}:2){};
\node[main node, label=below:$v$] (v6) at (0,0){};

\draw (v1)--(v2)--(v3)--(v4)--(v5)--(v1);
\draw (v1)--(v6); \draw (v2)--(v6); \draw (v3)--(v6); \draw (v4)--(v6); \draw (v5)--(v6);
\end{tikzpicture}
\caption{The graph $W_5$, the only graph on at most $6$ vertices with no semi-transitive orientation.}
\label{wheel5}
\end{figure}
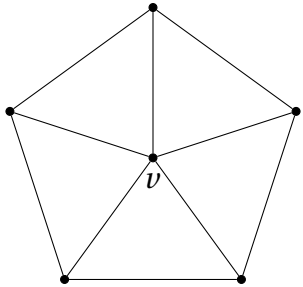

\subsection{Relation with word-representable graphs}\label{sec:word}

The concept of a semi-transitive orientation is a crucial ingredient in the characterization of word-representable graphs. 
Given a set $S$, a \df{word} over $S$ is a sequence of elements, or \df{letters}, of $S$. 
Given a word $W$ over $S$, we say two distinct letters $u$ and $v$ are \df{alternating} if $uu$ and $vv$ do not appear in the subword of $W$ induced by $u$ and $v$.  
A graph $G$ is \df{word-representable} if there exists a word $W$ over $V(G)$ such that for two vertices $u,v\in V(G)$, $uv \in E(G)$ if and only if $u$ and $v$ are alternating in $W$.
Word-representable graphs are also known as \df{representable graphs} or \df{alternation graphs}.

First introduced by Kitaev in 2004, word-representable graphs are a nice connection of graph theory and combinatorics on words, as word-representable graphs can be encoded as one-dimensional strings. 
It is related to various fields of (discrete) mathematics, as its motivation comes from the study of the {\em Perkins semigroups}, which lies in the core of semigroup theory.
For more details, see~\cite{KS2008}.

In \cite{HKP2016}, Halld\'{o}rsson, Kitaev, and Pyatkin characterized word-representable graphs in terms of semi-transitive orientations.
This characterization is also practical in the sense that it can be utilized to determine if a given graph is word-representable or not.
We end this subsection with presenting this characterization. 
For a comprehensive overview on the theory of word-representable graphs, see~\cite{Kit2017, KL2015}.

\begin{theorem}[\cite{HKP2016}]\label{stoiffwr}
A graph is word-representable if and only if it has a semi-transitive orientation.
\end{theorem}

\subsection{Graph operations and semi-transitive orientations}
Studying semi-transitive orientations is evidently related to understanding graphs that have a semi-transitive orientation.
One natural approach is to investigate how certain graph operations affect a graph with a semi-transitive orientation.
In this viewpoint, it is natural to ask the following question for each well-known graph operation.

\begin{ques}\label{que1}
Does a given graph operation preserve semi-transitive orientability?
\end{ques}

The semi-transitive orientability of graphs is {\em hereditary} in the sense that if a graph $G$ is semi-transitively orientable, then every induced subgraph of $G$ is also semi-transitively orientable.
Therefore, Question~\ref{que1} is answered in the affirmative for vertex deletions. 
For completion, we give a short proof in terms of semi-transitive orientations.

\begin{prop}\label{vert-del}
For a semi-transitively orientable graph $G$ and a vertex $v\in V(G)$, the graph $G -v$ is also semi-transitively orientable.
\end{prop}
\begin{proof}
Let $\phi$ be a semi-transitive orientation of $G$.
Suppose to the contrary that there is a vertex $v\in V(G)$ such that $G-v$ is not semi-transitively orientable.
Since $\phi(G-v)$ is not a semi-transitive orientation, and it clearly contains no directed cycles, there must be a shortcut $T$ in $\phi(G-v)$.
Yet, this shortcut $T$ is also a shortcut in $\phi(G)$, which contradicts the assumption that $\phi$ is a semi-transitive orientation of $G$. 
Hence, if $G-v$ is not semi-transitively orientable, then $G$ is also not semi-transitively orientable.
\end{proof}

There are many results regarding various graph operations on word-representable graphs.
Since Theorem~\ref{stoiffwr} reveals the relation between word-representable graphs and semi-transitive orientations, some of these results give a direct answer to Question~\ref{que1} for certain graph operations.
For instance, connecting two graphs by an edge, gluing two graphs at a vertex, and taking the Cartesian product of two graphs always preserve the semi-transitive orientability~\cite{Kit2017}.
On the other hand, even the most basic graph operations such as edge-deletion, edge-addition, and edge-contraction do not necessarily preserve this property.
We give simple examples here by starting with a graph that is semi-transitively orientable, and repetitively applying a certain graph operation to end up with $W_5$, which is not semi-transitively orientable as stated in Theorem~\ref{thm:6vx}.
Note that complete graphs, empty graphs, and bipartite graphs are semi-transitively orientable; see~\cite{Kit2017} for a proof. 

\begin{itemize}
\item {\bf Edge-deletion}\\
Start with the complete graph on $6$ vertices, and delete the appropriate edges to end up with $W_5$.

\item {\bf Edge-addition}\\
Start with the empty graph on $6$ vertices, and add the appropriate edges to end up with $W_5$.

\item {\bf Edge-contraction}\\
Start with the bipartite graph obtained by replacing each edge in $W_5$ with a path of length $2$, and contract each edge with an endpoint of degree $2$ to end up with $W_5$.
\end{itemize}

We refer the reader to~\cite{Kit2017, KL2015} for more results regarding graph operations.

\subsection{Our contributions}

As explained in the previous subsection, the answer to Question~\ref{que1} is known for various graph operations, and unfortunately many graph operations yield a negative answer.
In~\cite{KL2015}, it was explicitly asked if certain graph products preserve the semi-transitive orientability.
In particular, the tensor product and lexicographical product were mentioned.
In Section~\ref{sec:prod}, we answer the question in the negative for the tensor product, lexicographic product, and strong product by providing explicit examples.  
We will define the graph products in Section~\ref{sec:prod}. 

We modify Question~\ref{que1} slightly, and instead of asking if a certain graph operation always preserves the semi-transitive orientability, we ask if there exists an element of the graph such that the graph operation can be executed on while preserving the semi-transitive orientability. 
We state this as Question~\ref{que-main2}.

\begin{ques}\label{que-main2}
Given a graph operation and a semi-transitively orientable graph $G$, can the graph operation be executed while preserving the semi-transitive orientability?
\end{ques}

The main results of this paper is an affirmative answer to Question~\ref{que-main2} for edge-deletions, edge-additions, and edge-liftings. 
We were able to show that if a semi-transitively orientable graph has an edge $e$ such that the given graph operation is possible, then the graph has an edge (not necessarily $e$) such that the graph operation can be executed on while preserving the semi-transitive orientability. 
Moreover, for all three graph operations, we show that the initial semi-transitive orientation can be extended to the new graph obtained by the graph operation. 

\begin{theorem}\label{toempty}
If $\phi$ is a semi-transitive orientation of a semi-transitively orientable graph $G$ that is not the empty graph, then there exists an edge $e$ such that $\phi(G\setminus e)$ is a semi-transitive orientation of $G\setminus e$.
\end{theorem}
\begin{theorem}\label{tocomplete}
If $\phi$ is a semi-transitive orientation of a semi-transitively orientable graph $G$ that is not the complete graph, then there exist two non-adjacent vertices $u$ and $v$ such that $\phi$ can be extended to a semi-transitive orientation of $G+uv$. 
\end{theorem}
\begin{theorem}\label{tomatching}
If $\phi$ is a semi-transitive orientation of a semi-transitively orientable graph $G$ with maximum degree at least $2$, then there exists a path $P$ of length $2$ such that $\phi$ can be extended to a semi-transitive orientation of the graph obtained from $G$ by lifting $P$.
\end{theorem}

We also push the investigation further and ask the following question that initiates the study of sufficient conditions that guarantee a certain graph operation to preserve the semi-transitive orientability.

\begin{ques}\label{que-main}
When does a graph operation preserve the semi-transitive orientability?
\end{ques}

We were able to obtain results regarding Question~\ref{que-main} for edge-deletions, edge-subdivisions, and edge-additions. 
Given an edge $e$ in a graph $G$, we find an interesting relation between the graph obtained by deleting $e$ and subdividing $e$. 
We also give two sufficient conditions on edges so that edge-deletions and edge-subdivisions on those edges always preserve the semi-transitive orientability.

\begin{theorem}\label{ediffes}
Let $e$ be an edge of a graph $G$.
The graph obtained from $G$ by subdividing $e$ at least once is semi-transitively orientable if and only if $G\setminus e$ is semi-transitively orientable.
\end{theorem}
\begin{theorem}\label{k4free}
Let $e$ be an edge of a semi-transitively orientable graph $G$. 
If there is no $K_4$-subgraph of $G$ that contains $e$, then deleting $e$ or subdividing $e$ preserves the semi-transitive orientability of $G$.
\end{theorem}

We show an intriguing result that edge-additions are not guaranteed to preserve the semi-transitive orientability even if we have the additional condition that the graph has no short odd cycle.

\begin{theorem}\label{oddcycle}
For each positive integer $k$, there exists a semi-transitively orientable graph $G$ with no odd cycle of length at most $2k+1$ such that $G$ has a pair of non-adjacent vertices $u$ and $v$ where $G + uv$ is not semi-transitively orientable.
\end{theorem}


Section~\ref{sec2} concerns our results regarding edge-deletions and edge-subdivisions.
Our results related to edge-additions and edge-lifitings are in Section~\ref{sec3} and Section~\ref{sec:lift}, respectively. 
Section~\ref{sec:prod} is about graph products. 
We conclude this paper with Section~\ref{sec4} by presenting interesting open problems and suggesting future research directions.

\section{Edge-deletions and edge-subdivisions}\label{sec2}

We start this section with a proof of Theorem~\ref{ediffes}.
One direction can be easily shown by the following argument:
For an edge $e$ of a graph $G$, let $G_e$ be the graph obtained from $G$ by subdividing $e$ at least once, namely, $e$ is replaced by a path of length at least $2$. 
Suppose $G_e$ is semi-transitively orientable.
Now, $G\setminus e$ is also semi-transitively orientable by Proposition~\ref{vert-del} since $G\setminus e$ is exactly the induced subgraph of $G_e$ on $V(G)=V(G\setminus e)$.
Hence, it remains to prove the opposite direction.


\begin{theorem}
Let $e$ be an edge of a graph $G$.
If $G\setminus e$ is semi-transitively orientable, then the graph $G_e$ obtained from $G$ by subdividing $e$ at least once is also semi-transitively orientable.
\end{theorem}
\begin{proof}
Fix a semi-transitive orientation $\phi$ of $G\setminus e$ and let $e=xy$.
We will show that one can extend the semi-transitive orientation $\phi$ of $G\setminus e$ to a semi-transitive orientation of $G_e$.
Suppose that $G_e$ is obtained from $G$ by replacing the edge $e$ with a path $P=x,p_1,\dots,p_t,y$, for $t \geq 1$.
Let $\phi_e$ be an acyclic orientation of $G_e$ such that each edge of $G\setminus e$ follows the orientation of $\phi$ and the orientations of the edges on $P$ are $x\from p_1 \to \cdots\to p_t\to y$.
Note that $\phi_e$ is still an acyclic orientation of $G_e$ since $\phi_e(P)$ is not a directed path. 
If $\phi_e$ is not a semi-transitive orientation of $G_e$, then there exists a shortcut $S$ in $\phi_e(G_e)$. 
Therefore, there is a directed path either $Q=x\from q_1\from \cdots \from q_s\from y$
where $V(Q)\subset V(G\setminus e)$ and $V(S)=V(P)\cup V(Q)$ or $Q'=x\to q'_1\to\cdots\to q'_{s'}\to y$ where $V(Q')\subset V(G\setminus e)$, $V(S)=V(P)\cup V(Q')$, and $t=1$.
Note that $Q$ and $Q'$ cannot both exist since $V(Q)\cup V(Q')$ would contain a directed cycle in $\phi(G\setminus e)$. 

Suppose that $Q$ exists but $Q'$ does not.
We claim that changing the orientation of $P$ to $x\from p_1 \from \cdots\from p_t\from y$ gives a semi-transitive orientation of $G_e$. 
If not, then for some positive integer $k$, there must exist a path $X=x,x_1,\dots,x_{k-1},y$ ($x=x_0,y=x_k$) in $G\setminus e$ where the orientation of $X$ in $G\setminus e$ is $x_{j-1}\from x_j$ for some $j \in [k]$ and $x_{i-1}\to x_i$ for each $i\in [k]\setminus\{j\}$.
Also, the sets of vertices of $X$ and $Q$ have only $x$ and $y$ in common. 
Otherwise, let $i_0$ and $i_1$ be the smallest and the largest, respectively, $i \in [k]$ such that $x_i=q_j$ for some $j \in [s]$; let $x_{i_0}=q_{j_0}$ and $x_{i_1}=q_{j_1}$. 
Now, $G\setminus e$ contains a cycle $x \to x_1 \to \dots \to x_{i_0}=q_{j_0} \to q_{j_0-1} \to \dots \to q_1 \to x$ or $y \to q_s \to \dots \to q_{j_1}=x_{i_1} \to x_{i_1+1} \to \dots \to x_{k-1} \to y$, which contradicts that $\phi_e$ is acyclic. 
Thus, $V(X)\cap V(Q)=\{x, y\}$, and therefore, the vertices of $X$ and $Q$ form a shortcut in $G \setminus e$ since the edge $e=xy$ is not in $G \setminus e$. 
Hence, such a path $X$ does not exist in $G \setminus e$.

Now, assume that $t=1$ and $Q'$ exists but $Q$ does not. 
We claim that changing the orientation of $P$ to $x\to p_1 \to y$ gives a semi-transitive orientation of $G_e$.
If not, then for some positive integer $k$, there exists a path $X=x,x_1,\dots,x_{k-1},y$ ($x=x_0,y=x_k$) in $G\setminus e$ where the orientation of $X$ in $G\setminus e$ is $x_{j-1}\to x_j$ for some $j \in [k]$ and $x_{i-1}\from x_i$ for each $i\in [k]\setminus\{j\}$. 
By similar arguments as above, the vertices of $X$ and $Q'$ would contain a cycle or form a shortcut in $G\setminus e$. 
Hence, such a path $X$ does not exist in $G\setminus e$.

In all cases, we extended a semi-transitive orientation $\phi$ of $G\setminus e$ to a semi-transitive orientation of $G_e$. 
\end{proof}

Now, we show two different sufficient conditions on an edge $e$ such that deleting $e$ preserves the semi-transitive orientability of the graph. 
By Theorem~\ref{ediffes}, it is enough to show Theorem~\ref{k4free} for edge-deletions.

\begin{theorem}\label{k4del}
Let $e$ be an edge of a semi-transitively orientable graph $G$.
If $e$ is not contained in a $K_4$, then $G\setminus e$ is also semi-transitively orientable.
\end{theorem}
\begin{proof}
Let $\phi$ be a semi-transitive orientation of $G$ and let $e=xy$.
Assume that $e$ is not in a $K_4$-subgraph of $G$.
Suppose to the contrary that there is a shortcut $S$ in $\phi(G\setminus e)$, which further implies that $|V(S)|\geq 4$. 
Since $\phi(G)$ does not contain a shortcut, it must be that $G[V(S)]$ is a complete graph of order at least $4$ and $\phi(G[V(S)])$ is transitive.
Since $S$ is a shortcut in $\phi(G\setminus e)$, it must be that $x,y \in V(S)$.
Now, $e$ belongs to a complete subgraph of $G$ of order at least $4$, which contradicts the assumption on $e$.
\end{proof}

A \df{diamond} is the graph obtained from $K_4$ by removing an edge; see Figure~\ref{diamond} for an illustration.
The following corollary regarding edge-additions in graphs with no induced diamond is a direct consequence of Theorem~\ref{k4free}.
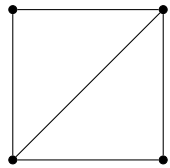
\begin{figure}[htbp]
\centering
\begin{tikzpicture}[main node/.style={fill,circle,draw,inner sep=0pt,minimum size=3pt}, scale=2]
\node[main node] (v1) at (0,0){};
\node[main node] (v2) at (1,0){};
\node[main node] (v3) at (0,1){};
\node[main node] (v4) at (1,1){};
\draw (v1)--(v2)--(v4)--(v3)--(v1)--(v4);
\end{tikzpicture}
\caption{A diamond.}
\label{diamond}
\end{figure}

\begin{cor}
Let $G$ be a graph that does not have a diamond as an induced subgraph.
If $G$ is not semi-transitively orientable,
then for every two non-adjacent vertices $u$ and $v$, the graph obtained by adding the edge $uv$ to $G$ is still not semi-transitively orientable.
\end{cor}

Note that Theorem~\ref{k4del} implies that if $G$ is a semi-transitively orientable graph that does not have $K_4$ as a subgraph, then $G \setminus e$ is also semi-transitively orientable for every edge $e$ in $G$.
This statement is not true if we replace $K_4$ with $K_5$.
For example, the graph in Figure~\ref{k5free} is semi-transitively orientable by Theorem~\ref{thm:6vx}, but after deleting the edge $uv$, we obtain $W_5$, which is not semi-transitively orientable.
\begin{figure}[htbp]
\centering
\begin{tikzpicture}[main node/.style={fill,circle,draw,inner sep=0pt,minimum size=3pt}, scale=1]
\node[main node] (v1) at ({90+72*0}:2){};
\node[main node,label=left:$u$] (v2) at ({90+72*1}:2){};
\node[main node] (v3) at ({90+72*2}:2){};
\node[main node] (v4) at ({90+72*3}:2){};
\node[main node,label=right:$v$] (v5) at ({90+72*4}:2){};
\node[main node] (v6) at (0,0){};

\draw (v1)--(v2)--(v3)--(v4)--(v5)--(v1);
\draw (v1)--(v6); \draw (v2)--(v6); \draw (v3)--(v6); \draw (v4)--(v6); \draw (v5)--(v6);
\draw (v2)--(v5);
\end{tikzpicture}
\caption{Deleting the edge $uv$ gives $W_5$.}
\label{k5free}
\end{figure}
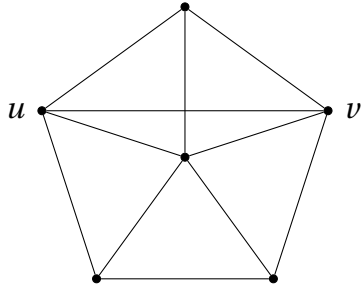

On the other hand, for every semi-transitively orientable graph, even if it contains $K_4$ as a subgraph, we can always find an edge $e$ where deleting $e$ preserves the semi-transitive orientability, unless the graph has no edges. 
For instance, by Theorem~\ref{thm:6vx}, we know that an edge-deletion of any edge other than $uv$ in the graph in Figure~\ref{k5free} preserves the semi-transitive orientability.

We now prove Theorem~\ref{toempty}. 
We actually prove Theorem~\ref{thm:strongtoempty}, which is a stronger result stating that we can remove a certain edge. 
A \df{source} in a directed graph is a vertex $u$ where every directed edge $uv$ incident with $u$ is oriented $u\to v$. 
A \df{sink} in a directed graph is a vertex $v$ where every directed edge $uv$ incident with $v$ is oriented $u\to v$. 
One can observe the following:

\begin{itemize}
\item If a directed graph contains no directed cycle, then it has at least one sink.
\item For every transitive orientation $\phi$ of a complete graph $K$, the directed graph $\phi(K)$ has exactly one source and exactly one sink.
\item	
Suppose $K$ is a complete graph of order at least $4$.
For every transitive orientation $\phi$ of $K$, the directed graph $\phi(K\setminus e)$ is not a shortcut if $e$ is an edge that connects the vertices corresponding to the source and the sink of $\phi(K)$.
\end{itemize}

\begin{theorem}\label{thm:strongtoempty}
Let $\phi$ be a semi-transitive orientation of a non-empty semi-transitively orientable graph $G$.
If $x_n$ is a sink of $\phi(G)$, $K$ is a maximal complete subgraph of $G$ containing $x_n$, and $x_1$ is a source of $\phi(K)$, then $G\setminus x_1 x_n$ is semi-transitively orientable.
\end{theorem}

We claim that the orientation $\phi$ of $G\setminus x_1x_n$ is still semi-transitive.
It is clear when $K$ is a complete graph on {\color{red}$2$} vertices, so we may assume that $K$ has order at least {\color{red}$3$}.
The following lemma is crucial.

\begin{claim}\label{clm:sourcesink}
If $K' \neq K$ is a complete subgraph of $G$ containing $x_1$ and $x_n$,
then $x_n$ is a sink of $\phi(K')$ and $x_1$ is a source of $\phi(K')$.
\end{claim}

\begin{proof}
It is not hard to see when $K'\subset K$, so we may assume $K'\nsubseteq K$.
It is clear that $x_n$ is the sink of $\phi(K')$ since $x_n$ is a sink of $\phi(G)$.
Suppose to the contrary that some vertex $y\in V(K')$ such that $y \neq x_1$ is the source of $\phi(K')$.
Note that $y\neq x_n$ since otherwise that would imply $K'$ is a single vertex. 
Since $y$ is the source of $K'$, the directed edges $y\to x_1$ and $y\to x_n$ exist. 
Let $K=\{x_1, \ldots, x_n\}$ and without loss of generality, we may assume that $x_1\to \cdots \to x_n$ is the longest directed path in $\phi(K)$.
Consider the directed path $y \to x_1 \to \cdots\to x_n$.
Since $y\to x_n$ exists and $\phi$ is a semi-transitive orientation of $G$, it must be that $G[\{x_1,\dots,x_n,y\}]$ is a complete subgraph of $G$.
Yet, this contradicts the maximality of $K$.
\end{proof}


\begin{proof}[Proof of Theorem~\ref{thm:strongtoempty}]
 Suppose to the contrary that there is a shortcut $S$ in $\phi(G\setminus x_1 x_n)$.
 Since $\phi(G)$ does not contain a shortcut, it must be that $G[V(S)]$ is a complete graph and $\phi(G[V(S)])$ is transitive.
We know that $x_n$ is the sink of $S$ since $x_n$ is a sink of $\phi(G)$. 
 On the other hand, by the above, it must be that $x_1$ is not the source of $S$ since $S$ is a shortcut in $\phi(G\setminus x_1 x_n)$.
 This contradicts Claim~\ref{clm:sourcesink}.
\end{proof}

Thus, starting from a semi-transitively orientable graph $G$ that is not the empty graph, one can obtain the empty graph by applying a sequence of edge-deletions while preserving the semi-transitive orientability.


\section{Edge-additions}\label{sec3}

In this section, we focus on edge-additions.
Unlike edge-deletions, edge-additions are not guaranteed to preserve the semi-transitive orientability even for graphs that do not contain $K_4$ as subgraphs.
This is because of the following result by Halld\'{o}rsson, Kitaev, and Pyatkin~\cite{HKP2011}.

\begin{theorem}[\cite{HKP2011}]\label{triangle}
For each positive integer $k$, there exists a graph with no $K_3$-subgraph that is not semi-transitively orientable.
\end{theorem}

The key ingredient in the proof of Theorem~\ref{triangle} is the following result by Erd\H{o}s \cite{Erd59}.
Recall that the \df{girth} of a graph is the length of a shortest cycle in the graph, and the \df{chromatic number} of a graph $G$ is the minimum positive integer $k$ such that $V(G)$ can be partitioned into $k$ independent sets. 
\begin{theorem}[\cite{Erd59}]\label{erdos}
For each positive integers $k$ and $l$, there exists a graph with chromatic number at least $k$ and girth at least $l$.
\end{theorem}

Now we give a proof of Theorem~\ref{oddcycle}.
It is sufficient to show the following generalization of Theorem~\ref{triangle}.
Note that if we remove an edge from a graph with no odd cycle of length at most $2k+1$, then the resulting graph also has no odd cycle of length at most $2k+1$. Thus if we take an edge-minimal graph $G$ that satisfies the condition in Theorem~\ref{odd}, then for each edge $e$ in $G$, the graph $G-e$ satisfies the condition in Theorem~\ref{oddcycle}.
The proof of Theorem~\ref{odd} can be done by utilizing Theorem~\ref{erdos}.

\begin{theorem}\label{odd}
For each positive integer $k$, there exists a graph with no odd cycle of length at most $2k+1$ that is not semi-transitively orientable.
\end{theorem}
\begin{proof}
Let $H$ be a graph with chromatic number at least $4$ and girth at least $6k+4$, for $k \geq 1$.
Define a supergraph $G$ of $H$ whose vertex set is the same as that of $H$ and edge set is defined as 
$$E(G) = E(H) \cup \{xy: x,y\in V(G)\text{ and there is a path from }x\text{ to }y\text{ of length }3\text{ in }H\}.$$
We claim that $G$ does not contain an odd cycle of length at most $2k+1$ and $G$ is not semi-transitively orientable.

First assume that $G$ contains an odd cycle $C$ of length $2m+1$ where $m\in[k]$.
By the definition of $G$, for each edge $uv$ in $C$, the distance between $u$ and $v$ in $H$ is either $1$ or $3$.
Consider a closed walk $C'$ in $H$ that is obtained from $C$ by replacing each edge in $E(C)\setminus E(H)$ with the corresponding path of length $3$ in $H$.
It is clear that the length of $C'$ is odd and is at most $3\times(2m+1) = 6m+3 < 6k+4$.
Hence $C'$ contains an odd cycle of length less than $6k+4$, which contradicts the assumption on the girth of $H$.

Now we will show that every acyclic orientation of $G$ has a shortcut.
Fix an acyclic orientation $\phi$ of $G$.
Note that $H$ is a subgraph of $G$. 
If $\phi(H)$ has no directed path of length $3$, then the chromatic number of $H$ is at most $3$, which contradicts that $H$ has chromatic number at least $4$. 
Thus there is a directed path of length $3$ in $\phi(H)$, say $v_1\to v_2\to v_3\to v_4$.
Then by the definition of $G$, $v_1 v_4$ is an edge in $G$ and is oriented as $v_1\to v_4$ in $\phi(G)$ because $\phi$ is acyclic.
On the other hand, $v_1$ and $v_3$ are not adjacent in $G$ because the girth of $H$ is at least $10$.
This implies that the induced subgraph of $\phi(G)$ on $\{v_1,v_2,v_3,v_4\}$ is a shortcut.
\end{proof}



In an arbitrary graph, edge-additions are possible, unless the graph was already a complete graph. 
In this section, we prove Theorem~\ref{tocomplete}, which states that we can keep adding edges to a semi-transitively orientable graph until the graph itself becomes a complete graph. 
We actually prove a stronger statement that we can keep extending the initial semi-transitive orientation. 
We divide the proof into 3 steps. 

\begin{enumerate}[Step 1.]
\item Start from a semi-transitively orientable graph $G$.
Add edges one-by-one to $G$ so that each edge addition preserves the semi-transitive orientability and the resulting graph $H_1$ is a comparability graph.
\item Add edges one-by-one to the graph $H_1$ so that each edge addition preserves the transitively orientable property and the resulting graph $H_2$ is a complete multipartite graph.
\item Add edges one-by-one to the graph $H_2$ so that each edge addition preserves the transitively orientable property and the resulting graph $K$ is a complete graph.
\end{enumerate}

We prove each step as a lemma.
In each lemma, we will use the following partition of the vertex set of the given graph $G$.
A \df{good} $\{V_i\}_{i\in[m]}$-\df{partition} of $V(G)$ is a partition of $V(G)$ into $m$ sets $V_1, \ldots, V_m$ such that $V_1$ is the set of all sources of $\phi(G)$ and for each $i\in[ m-1]$, $V_{i+1}$ is the set of all sources of $\phi\left(G\setminus \bigcup_{j\in[i]}V_j\right)$.
Note that for each $i \in [m]$, the induced subgraph $G[V_i]$ of $G$ on $V_i$ is the empty graph.

\begin{lemma}\label{semi-compar}
Let $\phi$ be a semi-transitive orientation of a semi-transitively orientable graph $G$. 
If $G$ is not a comparability graph, then it has a pair of non-adjacent vertices $u$ and $v$ such that $\phi$ can be extended to a semi-transitive orientation of $G+xy$. 
\end{lemma}
\begin{proof}
Assume $G$ is not a comparability graph, which means that $\phi$ cannot be a transitive orientation of $G$ since $G$ is not a comparability graph.
Hence, there exist three vertices $u, v, w$ of $G$ such that there is a directed path $u\to v \to w$ in $\phi(G)$ but there is no directed edge $u \to w$ in $\phi(G)$.


Consider a good $\{V_i\}_{i\in[m]}$-partition of $V(G)$. 
Choose a pair of two non-adjacent vertices $u$ and $w$ in $G$ with minimum $k-i$ among all pairs satisfying the following two properties.
\begin{enumerate}[(i)]
\item $u \in V_i$ and $w \in V_k$ where $1 \leq i < k-1 \leq m-1$.
\item There exists a vertex $v \in V_j$ where $i < j < k$ such that $uv, vw \in E(G)$.
Note that those edges are oriented as $u\to v$ and $v\to w$ in $\phi(G)$.
\end{enumerate}
We claim that the graph $G+uw$ obtained from $G$ by adding the edge $uw$ has a semi-transitive orientation $\psi$ such that $\psi(G) = \phi(G)$ an $\psi(uw) = u\to w$.

Suppose there is a shortcut $S$ in $\psi(G+uw)$.
Note that $S$ must contain $u\to w$.
If $u\to w$ is part of the longest path in $S$, then $\phi(G)$ already contained a shortcut, namely, the directed graph $S'$ obtained from $S$ by replacing the directed edge $u\to w$ with the directed path $u\to v\to w$.
Therefore, it must be that $u$ and $w$ is the source and sink, respectively, of $S$.
Let $u\to v_1 \to \cdots \to v_t \to w$ be the longest directed path in $S$.
For each $z\in[t]$, the vertex $v_z$ is contained in $V_\beta$ for some $i <\beta < k$.
Hence, the minimality of $k-i$ implies that $u \to v_z$ and $v_z\to w$ exist for every $z\in[t]$, and $v_r\to v_s$ for every $1\leq r<s\leq t$ in $\phi(G)$.
This contradicts the assumption that $S$ is a shortcut; $S$ is a transitive orientation of a complete graph.
\end{proof}

\begin{lemma}\label{complmult}
Let $\phi$ be a transitive orientation of a comparability graph $G$. 
If $G$ is not a complete multipartite graph, then it has a pair of non-adjacent vertices $x$ and $y$ such that $\phi$ can be extended to a transitive orientation of $G+xy$. 
\end{lemma}
\begin{proof}
Consider a good $\{V_i\}_{i\in[m]}$-partition of $V(G)$. 
Among all vertices of $G$, let $v \in V_k$ with maximum $k$ such that the following holds.

\begin{itemize}
\item Every vertex $w\in V_{k'}$ with $k' > k$ is adjacent to every vertex that is not in $V_{k'}$.
\item There is a vertex $u \in V_{k'}$ with $k'<k$ where $u$ and $v$ are not adjacent.
\end{itemize}

Among all vertices of $G$ that are not adjacent to $v$, let $u \in V_{k'}$ be a vertex with $k'<k$ and minimum $k'$.
Note that such a vertex exists by our choice of $v$. 
We claim that the directed graph $\phi(G+uv)$ obtained from $\phi(G)$ by adding $u\to v$ is a transitive orientation.

Since $\phi(G)$ is already transitive, it suffices to show that if there is a directed path $v_1 \to v_2 \to v_3$ in $\phi(G+uv)$ such that either $v_1 = u, v_2 = v$, and $v_3\not\in\{ u,v\}$ or $v_1\not\in\{u,v\}, v_2=u$, and $v_3=v$, then there is a directed edge $v_1 \to v_3$ in $\phi(G+uv)$.
If we have a directed path $u \to v \to v_3$ in $\phi(G+uv)$, then by the maximality of $k$, 
it must be that the directed edge $u \to v_3$ exists in $\phi(G)$.
If we have a directed path  $v_1 \to u \to v$ in $\phi(G+uv)$, then by the minimality of $k'$, 
it must be that the directed edge $v_1\to v$ exists in $\phi(G)$.
\end{proof}

\begin{lemma}\label{complt}
Let $\phi$ be a transitive orientation of a complete multipartite graph $G$. 
If $G$ is not a complete graph, then it has a pair of non-adjacent vertices $x$ and $y$ such that $\phi$ can be extended to a transitive orientation of $G+xy$. 
%
\end{lemma}
\begin{proof}
Consider a good $\{V_i\}_{i\in[m]}$-partition of $V(G)$.
If $G$ is not a complete graph, then for some $z\in[m]$, $V_z$ contains at least two vertices $u$ and $v$.
We will show that the orientation $\phi(G+uv)$ obtained from $\phi(G)$ by adding $u\to v$ is transitive.

Similarly as in the previous lemma, since $\phi(G)$ is already transitive, it is sufficient to show that if there is a directed path $v_1 \to v_2 \to v_3$ in $\phi(G+uv)$ such that either $v_1 = u, v_2 = v$, and $v_3\not\in\{u,v\}$ or $v_1\not\in\{ u,v\}, v_2=u$, and $v_3=v$, then there is a directed edge $v_1 \to v_3$ in $\phi(G+uv)$.
If we have a directed path $u \to v \to v_3$ in $\phi(G+uv)$, then the directed edge $u \to v_3$ exists in $\phi(G)$ because $v_3 \in V_{z'}$ for some $z' >z$.
If we have a directed path $v_1 \to u \to v$ in $\phi(G+uv)$, then the directed edge $v_1\to v$ exists in $\phi(G)$ because $v_1 \in V_{z'}$ for some $z'< z$.
\end{proof}

Thus, starting from a semi-transitively orientable graph $G$ that is not the complete graph, one can obtain the complete graph by applying a sequence of edge-additions while preserving the semi-transitive orientability.


%
%
%
%

\section{Edge-liftings}\label{sec:lift}

Let $P=uvw$ be a path of length $2$ of a graph $G$.
The graph obtained by \df{lifting} $uvw$ is the graph obtained from $G$ by removing the edges $uv$ and $vw$ and adding the edge $uw$, if $uw$ did not exist already.
If lifting $P$ preserves the semi-transitive orientability of $G$, then adding the edge $uw$ preserves the semi-transitive orientability of $G\setminus v$.
Recall that if $G$ is semi-transitively orientable, then so is $G\setminus v$ by the heredity property of semi-transitive orientations.

Note that a graph always has a path of length $2$ that can be lifted, unless the graph has maximum degree at most $1$. 
We show that a semi-transitively orientable graph always has a path of length $2$ that can be lifted, unless the graph itself has maximum degree at most $1$. 

Recall that 
given a graph $G$, a good $\{V_i\}_{i\in[m]}$-partition of $V(G)$ is a partition of $V(G)$ into $m$ sets $V_1, \ldots, V_m$ such that $V_1$ is the set of all sources of $\phi(G)$ and for each $i\in[ m-1]$, $V_{i+1}$ is the set of all sources of $\phi\left(G\setminus \bigcup_{j\in[i]}V_j\right)$.

\begin{theorem}\label{lift}
If $\phi$ is a semi-transitive orientation of a semi-transitively orientable graph $G$ with  maximum degree at least $2$, then it has a path $P$ of length $2$ such that $\phi$ can be extended to a semi-transitive orientation of the graph obtained from $G$ by lifting $P$. 
\end{theorem}
\begin{proof}
Consider a good $\{V_i\}_{i\in[m]}$-partition of $V(G)$. 
Suppose $m \geq 3$.
Choose a pair of two non-adjacent vertices $u$ and $w$ in $G$ with minimum $k-i$ among all pairs satisfying the following two properties.
\begin{enumerate}[(i)]
\item $u \in V_i$ and $w \in V_k$ where $1 \leq i < k-1 \leq m-1$.
\item There exists a vertex $v \in V_j$ where $i < j < k$ such that $uv, vw \in E(G)$.
Note that those edges are oriented as $u\to v$ and $v\to w$ in $\phi(G)$.
\end{enumerate}
Let $G'$ be the graph obtained from $G$ by lifting the path $uvw$.
We claim that $G'$ has a semi-transitive orientation $\psi$ such that $\psi(G'\setminus uw) = \phi(G'\setminus uw)$ and $\psi(uw) = u\to w$.

If there is a shortcut $S$ in $\psi(G')$, then $S$ must contain $u\to w$.
If $u\to w$ is part of the longest path in $S$, then $\phi(G)$ already contained a shortcut, namely, the directed graph $S'$ obtained from $S$ by either adding the directed path $u\to v\to w$ or replacing the directed edge $u\to w$ with the directed path $u\to v\to w$.
Therefore, it must be that $u$ and $w$ is the source and sink, respectively, of $S$. 
Let $P = u\to v_1 \to \cdots \to v_t \to w$ be the longest directed path in $S$.
By the minimality of $k-i$, it implies that we actually have $t = 1$, thus $P$ has length $2$.
Yet, this contradicts the fact that a shortcut has a directed path of length at least $3$.

Now suppose $m = 2$.
We claim that in this case lifting any path of length $2$ preserves the semi-transitive orientability.
Without loss of generality, assume that we have $u, w\in V_1$ and $v \in V_2$ such that $uv,vw\in E(G)$.
Let $G'$ be the graph obtained from $G$ by lifting the path $uvw$.
Now, any acyclic orientation $\psi$ satisfying $\psi(G'\setminus uw) = \phi(G'\setminus uw)$ is a semi-transitive orientation on $G'$.
It is not hard to see that $\psi(G')$ does not contain a shortcut since every edge other than $uw$ is oriented by $v_1 \to v_2$ where $v_1\in V_1$ and $v_2\in V_2$.
\end{proof}

Thus, starting from a semi-transitively orientable graph $G$, one can obtain a graph where each component is an edge or a single vertex by applying a sequence of edge-liftings while preserving the semi-transitive orientability.

\section{Graph products}\label{sec:prod}

Assume $G$ and $H$ are both semi-transitively orientable graphs. 
It is known that $G\square H$, the Cartesian product of $G$ and $H$, is also semi-transitively orientable~\cite{KL2015}.
We investigate the semi-transitive orientability for three graph products: the tensor product, lexicographic product, and strong product. 
By constructing explicit examples, we show that the aforementioned three graph products do not necessarily preserve the semi-transitive orientability. 
In particular, the case of tensor product and lexicographic product was explicitly asked in~\cite{KL2015}.
See Table~\ref{tab:prod} for definitions of the tensor product, lexicographical product, and strong product.

Recall that the wheel on $6$ vertices is not semi-transitively orientable, and every graph on at most $5$ vertices is indeed semi-transitively orientable by Theorem~\ref{thm:6vx}. 
In all constructions below, we use two graphs $G$ and $H$ that both have at most $5$ vertices to obtain a graph product that contains $W_5$ as an induced subgraph. 
We were able to find such examples for the tensor product, lexicographical product, and strong product.
See Figure~\ref{fig:tensor} and Figure~\ref{fig:lexstro} for illustrations.

\begin{table}[h]
\centering
\begin{tabular}{c||c|c|c}\label{tab:prod}
Graph product & denoted by & vertex set & edge condition\\
\hline\hline
Tensor & $G\times H$ & $V(G)\times V(H)$ & $u_1\sim v_1 \mbox{ and }u_2\sim v_2$\\
\hline
Lexicographical& $G\cdot H$ & $V(G)\times V(H)$ & $ u_1\sim v_1 $\\
&  &  &  or $(u_1=v_1 \mbox{ and } u_2\sim v_2)$\\
\hline
Strong & $G\boxtimes H$ & $V(G)\times V(H)$ & \,\quad $(u_1=v_1\mbox{ and }u_2\sim v_2) $\\
&  &  & or $ (u_1\sim v_1\mbox{ and }u_2=v_2) $\\
 &  &  & or $ (u_1\sim v_1\mbox{ and }u_2\sim v_2)$\\
\end{tabular}
\caption{Definitions of tensor product, lexicographical product, and strong product of two graphs $G$ and $H$. The edge condition is when $(u_1, u_2)\sim(v_1, v_2)$.}
\end{table}


\begin{figure}[h]
    \centering
    \includegraphics[scale=0.7]{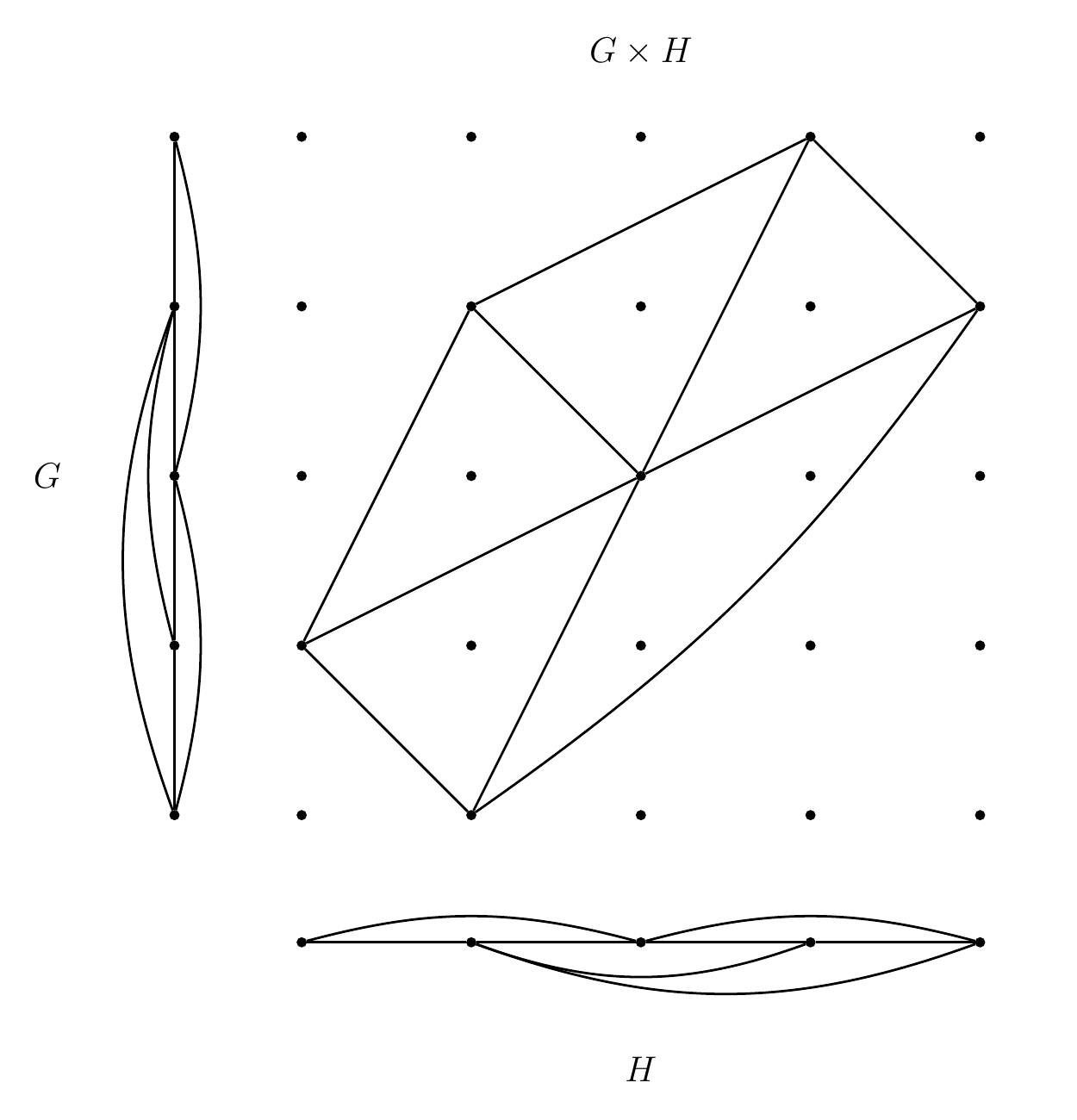}
    \caption{The tensor product of two semi-transitively orientable graphs that contains $W_5$ as an induced subgraph}
    \label{fig:tensor}
\end{figure}

\begin{figure}[h]
    \centering
    \includegraphics[scale=0.675]{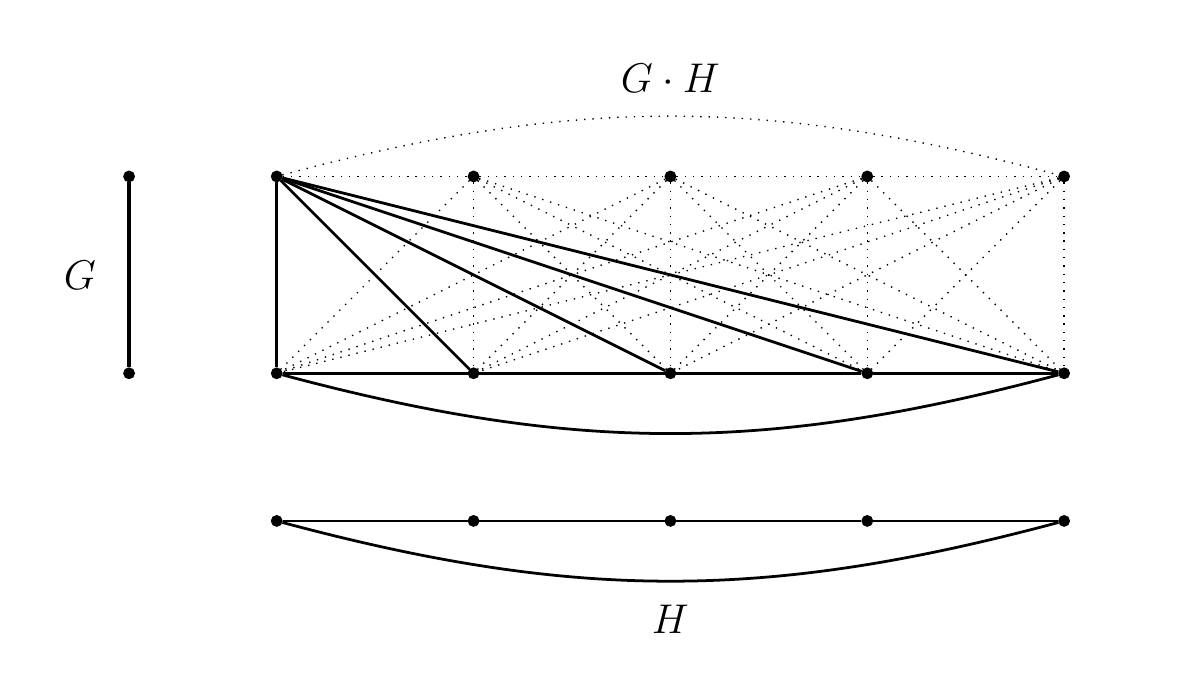}
    \includegraphics[scale=0.675]{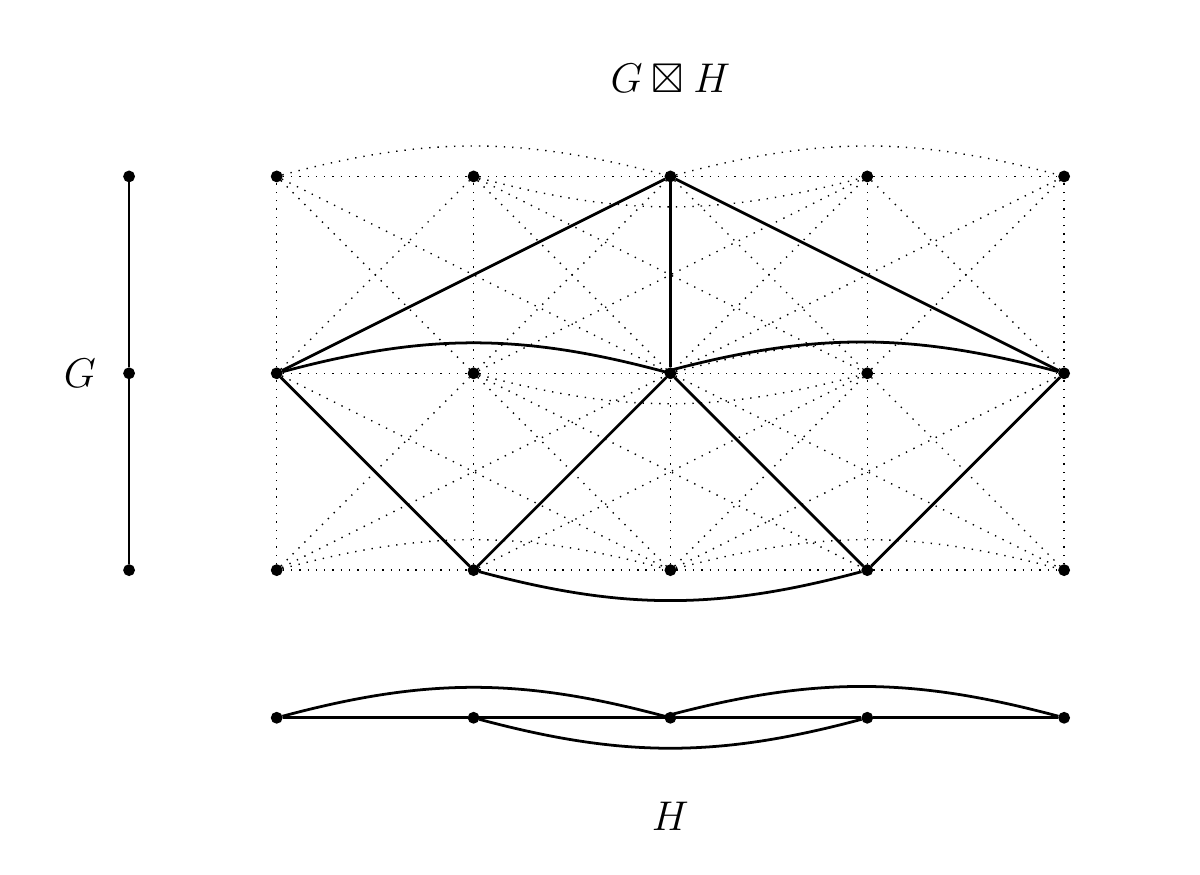}
    \caption{The lexicographic product (left) and the strong product (right) of two semi-transitively orientable graphs that contain $W_5$ as an induced subgraph}
    \label{fig:lexstro}
\end{figure}

\section{Concluding remarks}\label{sec4}

The main results in this paper show that for edge-deletions, edge-additions, and edge-liftings, if a certain graph operation can be executed on a semi-transitively orientable graph, then the same graph operation can be executed while preserving the semi-transitive orientability. 
For edge-deletions, edge-additions, and edge-liftings, this gives a positive answer to Question~\ref{que-main2}, which we reiterate here. 

\begin{ques2}
Given a graph operation and a semi-transitively orientable graph $G$, can the graph operation be executed while preserving the semi-transitive orientability?
\end{ques2}

It would be interesting to determine other graph operations that behave in the same manner. 
In particular, is Question~\ref{que-main2} true for edge-contractions?

\begin{ques}\label{que:contract}
Given a non-empty semi-transitively orientable graph $G$, is there an edge $e$ where the graph obtained from $G$ by contracting $e$ is also semi-transitively orientable?
\end{ques}

We know edge-contractions do not always preserve the semi-transitive orientability, as mentioned in the introduction. 
Finding sufficient conditions for edge-contractions to preserve the semi-transitive orientability would be a good starting point in investigating Question~\ref{que:contract}.
This would be in the spirit of Question~\ref{que-main}, which we repeat here. 

\begin{ques1}
When does a graph operation preserve the semi-transitive orientability?
\end{ques1}

Theorem~\ref{odd} shows the existence of a graph with no short odd cycles that does not have a semi-transitive orientation.
Yet, no information is gained regarding even cycles. 
Since $W_5$ is a graph with no even cycle of length at least $8$ that does not have a semi-transitive orientation, it is natural to ask the following two questions:

\begin{ques}
Is there a graph with no cycle of length $6$ that is not semi-transitively orientable?
\end{ques}
\begin{ques}
Is there a graph with no cycle of length $4$ that is not semi-transitively orientable?
\end{ques}


\bibliographystyle{alpha}

%

\end{document}